\documentclass{jms}
\paperTitle{Singular points of the integral representation of the Mittag-Leffler function}
\articleColonName{Singular point of the ML function}
\authorsShort{V.\,V.~Saenko}
\authorsFull{V.\,V. Saenko\first}
\addAuthorInfo{Ulyanovsk State University, S.P. Kapitsa Research Institute of Technology,   L. Tolstoy St. 42, Ulyanovsk,  Russia, 432017 e-mail: \url{saenkovv@gmail.com}}

\paperAbstract{The paper presents an integral representation of the two-parameter Mittag-Leffler function  $E_{\rho,\mu}(z)$ and singular points of this representation have been studied. It has been found that there are two singular points for this integral representation: $\zeta=1$ and $\zeta=0$. The point $\zeta=1$ is a pole of the first order and the point $\zeta=0$, depending on the values of parameters $\rho,\mu$ is either a pole or a branch point, or a regular point. The subsequent study showed that at some values of parameters $\rho,\mu$ with the help of the residue theory one can calculate the integral included in the studied integral representation and express the function  $E_{\rho,\mu}(z)$ through elementary functions.}

\begin{document}
\maketitle

\section{Introduction}

The Mittag-Leffler is an entire function which is defined by a power series\begin{equation*}
  E_{\rho}(z)=\sum_{k=0}^{\infty}\frac{z^k}{\Gamma(1+k/\rho)},\quad \rho>0,\quad z\in\C.
\end{equation*}
This function was introduced by Mittag-Leffler in a series of papers published from 1902-1905 in connection with his development of a method for summing divergent series. For more detailed information on the content of these works and on the history of the introduction of the Mittag-Leffler function, we refer the reader to the book \cite{Gorenflo2014} (see chapter 2 in \cite{Gorenflo2014}). The function $E_{\rho}(z)$ itself was introduced in the work \cite{Mittag-Leffler1903}.  An important generalization of this function
\begin{equation}\label{eq:MLF_gen}
  E_{\rho,\mu}(z)=\sum_{k=0}^{\infty}\frac{z^k}{\Gamma(\mu+k/\rho)},\quad \rho>0,\quad \mu\in\C,\quad z\in\C
\end{equation}
obtained by  A. Wiman  in 1905 \cite{Wiman1905,Wiman1905a} was developed in the works  \cite{zbMATH03082752,zbMATH03078845,zbMATH03081895}.

The interest in the Mittag-Leffler function is primarily associated with the use of this function for solving equations in fractional derivatives. At its core, this function is an eigenfunction of some operators of fractional differentiation and integration. This circumstance largely determines the scope of its use.   Numerous examples of the function use $E_{\rho,\mu}(z)$ for solving integral and differential equations of fractional order can be found in the books \cite{Gorenflo1997, Gorenflo2014, Gorenflo2019}.

The most important, from the point of view of studying asymptotic properties and from the point of view of numerical calculations, is the integral representation of the Mittag-Leffler function. There are several forms of writing the integral representation of the Mittag-Leffler function. One of the integral forms $E_{\alpha,\mu}(z)$ is given in the book \cite{Bateman_V3_1955} (see \S18.1) and it has the form
\begin{equation*}
  E_{\alpha,\mu}(z)=\frac{1}{2\pi i}\int_{C}\frac{\zeta^{\alpha-\mu}e^\zeta}{\zeta^\alpha-z}d\zeta,
\end{equation*}
where the contour of integration $C$ is the loop that starts and ends in $-\infty$ and embraces the circle $|t|\leqslant |\zeta|^{1/\alpha}$ in a positive direction: $-\pi\leqslant \arg t\leqslant\pi$. Here the parameter $\alpha$ is connected with the parameter $\rho$ in (\ref{eq:MLF_gen}) by the relation $\alpha=1/\rho$. Another form of integral representation was obtained in the work \cite{Dzhrbashyan1954_eng} and was studied in every detail in the book \cite{Dzhrbashyan1966_eng} (see chapter 3, \S 2 in \cite{Dzhrbashyan1966_eng}, and also \S 3.4 in \cite{Gorenflo2014}, and \S 1.3 in \cite{Dzhrbashian1993})
\begin{equation}\label{eq:MLF_int_Dzhr}
\begin{array}{c}
  \displaystyle E_{\rho,\mu}(z) =\frac{\rho}{2\pi i}\int_{\gamma}\frac{\exp\left\{\zeta^\rho\right\}\zeta^{\rho(1-\mu)}}{\zeta-z}d\zeta,\quad z\in G^{(-)}(\varepsilon,\delta),\\
  \displaystyle E_{\rho,\mu}(z) =\rho z^{\rho(1-\mu)}e^{z^\rho}+ \frac{\rho}{2\pi i}\int_{\gamma}\frac{\exp\left\{\zeta^\rho\right\}\zeta^{\rho(1-\mu)}}{\zeta-z}d\zeta,\quad  z\in G^{(+)}(\varepsilon,\delta),
\end{array}
\end{equation}
where the contour of integration $\gamma$ is composed of a half-line $\arg \zeta=-\delta$, $|\zeta|\geqslant\varepsilon$, an arc $-\delta\leqslant\arg \zeta\leqslant\delta$ a circle $|\zeta|=\varepsilon$ and a half-line $\arg\zeta=\delta$, $|\zeta|\geqslant\varepsilon$. Here $G^{(-)}(\varepsilon,\delta)$ and $G^{(+)}(\varepsilon,\delta)$ are the domains lying to the left and the right of the contour  $\gamma$, $\rho>1/2$, $\frac{\pi}{2\rho}<\delta\leqslant\min(\pi,\pi/\rho)$ respectively.

One more integral representation of the Mittag-Leffler function was presented in the works \cite{Saenko2019, Saenko2020}. In these works the following lemma was formulated and proved.

\begin{lemma}\label{lemm:MLF_int}  For any real $\rho, \delta_{1\rho}, \delta_{2\rho}, \epsilon$ which is such that $\rho>1/2$, $\frac{\pi}{2\rho}<\delta_{1\rho}\leqslant\min(\pi,\pi/\rho)$, $\frac{\pi}{2\rho}<\delta_{2\rho}\leqslant\min(\pi,\pi/\rho)$, $\epsilon>0$, any $\mu\in\C$ and any $z\in\C$ which is such that
\begin{equation}\label{eq:z_cond_lemm}
  \frac{\pi}{2\rho}-\delta_{2\rho}+\pi<\arg z<-\frac{\pi}{2\rho}+\delta_{1\rho}+\pi
\end{equation}
The Mittag-Leffler function can be represented in the form
\begin{equation}\label{eq:MLF_int}
  E_{\rho,\mu}(z)=\frac{\rho}{2\pi i} \int_{\gamma_\zeta}\frac{\exp\left\{(z\zeta)^{\rho}\right\}(z\zeta)^{\rho(1-\mu)}}{\zeta-1}d\zeta.
\end{equation}
where the contour of integration $\gamma_\zeta$ has the form
\begin{equation}\label{eq:loop_gammaZeta}
  \gamma_\zeta=\left\{\begin{array}{ll}
                       S_1=&\{\zeta: \arg\zeta=-\delta_{1\rho}-\pi,\quad |\zeta|\geqslant 1+\epsilon\},\\
                       C_\epsilon=&\{\zeta: -\delta_{1\rho}-\pi\leqslant\arg\zeta\leqslant\delta_{2\rho}- \pi,\quad |\zeta|=1+\epsilon\},\\
                       S_2=&\{\zeta: \arg\zeta=\delta_{2\rho}-\pi,\quad|\zeta|\geqslant1+\epsilon\}.
                     \end{array}\right.
\end{equation}
\end{lemma}
The proof of this lemma can be found in  \cite{Saenko2020}. The singular points of the integrand of the representation (\ref{eq:MLF_int}) are investigated in this work.  Knowing the singular points will make it possible to circumvent by deforming the integration contour $\gamma_\zeta$ and, thus, it will allow us to calculate the integral included in (\ref{eq:MLF_int}). As one can see, the integral included in this representation in the general case cannot be analytically calculated. However, special cases will be given below, in which using the theory of residues it is possible to calculate this integral and represent $E_{\rho,\mu}(z)$ through elementary functions. In all other cases it is necessary to use numerical methods to calculate this integral.

\section{Singular points of the Mittag-Leffler function }

We represent the expression (\ref{eq:MLF_int}) in the form  $E_{\rho,\mu}(z)=\int_{\gamma_\zeta}\Phi_{\rho,\mu}(\zeta,z)d\zeta$, where
\begin{equation}\label{eq:Phi_fun}
  \Phi_{\rho,\mu}(\zeta,z)=\frac{\rho}{2\pi i}\frac{\exp\{(\zeta z)^\rho\}(\zeta z)^{\rho(1-\mu)}}{\zeta-1}.
\end{equation}
Then, taking account of the notation introduced, for the integral representation (\ref{eq:MLF_int}) the following theorem is true

\begin{theorem}\label{lemm:MLF_SingPoints}
For any real $\rho>1/2$ and any complex values of the parameter $\mu=\mu_R+i\mu_I$ the function $\Phi_{\rho,\mu}(\zeta,z)$, relative to the variable $\zeta$, has two singular points $\zeta=1$ and $\zeta=0$. The point $\zeta=1$ is a pole of the first order. The point $\zeta=0$ is:
\begin{enumerate}
\item the regular point of the function $\Phi_{\rho,\mu}(\zeta,z)$, at values of parameters $\rho=n$, where $n=1,2,3,\dots$ (an  integer positive number), $\mu_I=0$ and $\mu_R=1-m_1/\rho$, where $m_1=0,1,2,3,\dots$ (an integer positive number);
\item the pole of the order $m_2$, if $\rho=n$, where $n=1,2,3,\dots$ (an integer positive number), $\mu_I=0$, and $\mu_R=1+m_2/\rho$, where $m_2=1,2,3,\dots$ (an integer positive number);
\item the branch point, for all other values of parameters $\rho, \mu_I, \mu_R$.
\end{enumerate}
\end{theorem}

\begin{proof}
To simplify the proof we represent the function $\Phi_{\rho,\mu}(\zeta,z)$ in the form
\begin{equation}\label{eq:Phi_product}
  \Phi_{\rho,\mu}(\zeta,z)=f_1(\zeta,z)f_2(\zeta,z)f_3(\zeta),
\end{equation}
where
\begin{equation*}
  f_1(\zeta,z)=\frac{\rho}{2\pi i}\exp\{(\zeta z)^{\rho}\},\quad
  f_2(\zeta,z)=(\zeta z)^{\rho(1-\mu)},\quad
  f_3(\zeta)=(1-\zeta)^{-1}
\end{equation*}
and investigate each of these functions in terms of singular points.

We consider the function $f_1(\zeta,z)$. We will represent the variables $\zeta$ and $z$ in the form $\zeta=re^{i\varphi}$, $z=te^{i\psi}$. As a result, we have
\begin{multline*}
  f_1(\zeta,z)\equiv f_1(r,\varphi,t,\psi)=\frac{\rho}{2\pi i}\exp\left\{(rt)^\rho e^{i\rho(\varphi+\psi)}\right\}
  =\frac{\rho}{2\pi i}\exp\left\{(rt)^\rho(\cos(\rho(\varphi+\psi))+i\sin(\rho(\varphi+\psi)))\right\}\\
  = \frac{\rho}{2\pi i}\exp\left\{(rt)^\rho \cos(\rho(\varphi+\psi))\right\}
  \times\left[\cos\left((rt)^\rho\sin(\rho(\varphi+\psi))\right)+i\sin\left((rt)^\rho\sin(\rho(\varphi+\psi))\right)\right]\\
  =f_1^{Re}(r,\varphi,t,\psi)+i f_1^{Im}(r,\varphi,t,\psi),
\end{multline*}
where
\begin{align*}
  f_1^{Re}(r,\varphi,t,\psi) & =\Re f_1(t,\psi,r,\varphi)
  =\frac{\rho}{2\pi}\exp\left\{(rt)^\rho \cos(\rho(\varphi+\psi))\right\}\sin\left((rt)^\rho\sin(\rho(\varphi+\psi))\right),\\
   f_1^{Im}(r,\varphi,t,\psi)& =\Im f_1(t,\psi,r,\varphi)
   = -\frac{\rho}{2\pi}\exp\left\{(rt)^\rho \cos(\rho(\varphi+\psi))\right\}\cos\left((rt)^\rho\sin(\rho(\varphi+\psi))\right).
\end{align*}
The only singular point in the variable $\zeta$ of the function $f_1(\zeta,z)$ can be the point $\zeta=0$, which can be a branch point. A power law  $\zeta^\rho$ testifies to this fact. By definition, a branch point is the point circumventing which along the closed contour that embraces this point leads to new elements of the function. The verification of the presence of a branch point in the point $\zeta=0$ of the function $f(\zeta)$ is similar to the verification of the equality implementation $f(|\zeta|,\arg\zeta)= f(|\zeta|,\arg\zeta+2\pi)$. If this equality is not implemented, then in the point $\zeta=0$ there is a branch point of the function $f(\zeta)$. Thus, for the function $f_1(\zeta,z)$, we have
\begin{multline*}
  f_1^{Re}(r,\varphi+2\pi,t,\psi)
  =\frac{\rho}{2\pi}\exp\left\{(rt)^\rho \cos(\rho(\varphi+\psi+2\pi))\right\}\sin\left((rt)^\rho\sin(\rho(\varphi+\psi+2\pi))\right)=\\
  \frac{\rho}{2\pi}\exp\left\{(rt)^\rho\left[ \cos(\rho(\varphi+\psi))\cos(2\pi\rho)- \sin(\rho(\varphi+\psi))\sin(2\pi\rho)\right]\right\}\times\\ \sin\left((rt)^\rho\left[\sin(\rho(\varphi+\psi))\cos(2\pi\rho)+\cos(\rho(\varphi+\psi))\sin(2\pi\rho)\right]\right).
\end{multline*}
From here it can be seen that
\begin{equation*}
  f_1^{Re}(r,\varphi,t,\psi)=f_1^{Re}(r,\varphi+2\pi,t,\psi), \quad\mbox{If}\quad \rho=n,
\end{equation*}
and
\begin{equation*}
  f_1^{Re}(r,\varphi,t,\psi)\neq f_1^{Re}(r,\varphi+2\pi,t,\psi), \quad\mbox{If}\quad \rho\neq n,
\end{equation*}
where $n=1,2,3,\dots$ (an integer positive number).

Similarly, for $f_1^{Im}(r,\varphi,t,\psi)$ we obtain
\begin{multline*}
  f_1^{Im}(r,\varphi+2\pi,t,\psi)
  =-\frac{\rho}{2\pi}\exp\left\{(rt)^\rho\left[ \cos(\rho(\varphi+\psi))\cos(2\pi\rho)- \sin(\rho(\varphi+\psi))\sin(2\pi\rho)\right]\right\}\times\\
  \cos\left((rt)^\rho\left[\sin(\rho(\varphi+\psi))\cos(2\pi\rho)+\cos(\rho(\varphi+\psi))\sin(2\pi\rho)\right]\right).
\end{multline*}
From here one can see that
\begin{equation*}
  f_1^{Im}(r,\varphi,t,\psi)=f_1^{Im}(r,\varphi+2\pi,t,\psi), \quad\mbox{If}\quad \rho=n
\end{equation*}
and
\begin{equation*}
  f_1^{Im}(r,\varphi,t,\psi)\neq f_1^{Im}(r,\varphi+2\pi,t,\psi), \quad\mbox{If}\quad \rho\neq n.
\end{equation*}

Thus, for the function $f_1(\zeta,z)$ the point
\begin{equation}\label{statement_f1}
  \zeta=0\ \mbox{ is}\ \left\{\begin{array}{cc}
                                        \mbox{the regular point}, & \mbox{If}\ \rho=n,\\
                                        \mbox{the branch point}, & \mbox{If}\ \rho\neq n,
                                      \end{array}\right.
\end{equation}
where $n=1,2,3,\dots$.

We consider now the function $f_2(\zeta,z)$. Representing the variables $\zeta$ and $z$ in the form $\zeta=re^{i\varphi}$, $z=te^{i\psi}$ and a complex parameter $\mu=\mu_R+i\mu_I$, we obtain
\begin{multline*}
  f_2(\zeta,z)\equiv f_2(r,\varphi,t,\psi)= (rt)^{\rho(1-\mu_R-i\mu_I)}\exp\left\{i\rho(1-\mu_R-i\mu_I)(\varphi+\psi)\right\}\\
  =(rt)^{\rho(1-\mu_R)}\exp\left\{\rho\mu_I(\varphi+\psi)\right\}\exp\left\{i[\rho(1-\mu_R)(\varphi+\psi)-\rho\mu_I\ln(rt)]\right\}\\
  =(rt)^{\rho(1-\mu_R)}\exp\left\{\rho\mu_I(\varphi+\psi)\right\}
  \left[\cos(\rho(1-\mu_R)(\varphi+\psi)-\rho\mu_I\ln(rt))\right.\\
  \left.+i\sin(\rho(1-\mu_R)(\varphi+\psi)-\rho\mu_I\ln(rt))\right]=
  f_2^{Re}(r,\varphi,t,\psi)+if_2^{Im}(r,\varphi,t,\psi),
\end{multline*}
where
\begin{align*}
  f_2^{Re}(r,\varphi,t,\psi) &=\Re f_2(r,\varphi,t,\psi)
  =(rt)^{\rho(1-\mu_R)}e^{\rho\mu_I(\varphi+\psi)} \cos(\rho(1-\mu_R)(\varphi+\psi)-\rho\mu_I\ln(rt)), \\
  f_2^{Im}(r,\varphi,t,\psi)&=\Im f_2(r,\varphi,t,\psi)
  =(rt)^{\rho(1-\mu_R)}e^{\rho\mu_I(\varphi+\psi)}
  \sin(\rho(1-\mu_R)(\varphi+\psi)-\rho\mu_I\ln(rt)).
\end{align*}

We will study the behavior of the function $f_2(r,\varphi,t,\psi)$ in the vicinity of the point  $\zeta=0$. From the expression for $f_2^{Re}(r,\varphi,t,\psi)$ and $f_2^{Im}(r,\varphi,t,\psi)$ it is clear that they contain a multiplier $(rt)^{\rho(1-\mu_R)}$. At values $\mu_R>1$ the index of power of this multiplier becomes negative. As a result, at $\zeta\to0$ this multiplier increases unlimitedly. Therefore, there is the pole in the point $\zeta=0$ of the function $f_2(\zeta,z)$ at $\mu_R>1$. It should be pointed out that the multiplier of the form $\cos(\ln(rt))$ or $\sin(\ln(rt))$ leads to the appearance of an oscillating function at $r\to0$ is limited in modulus. We will verify if the point $\zeta=0$ is a branch point. We have
\begin{multline*}
  f_2^{Re}(r,\varphi+2\pi,t,\psi)= (rt)^{\rho(1-\mu_R)}e^{\rho\mu_I(\varphi+\psi+2\pi)} \cos(\rho(1-\mu_R)(\varphi+\psi+2\pi)-\rho\mu_I\ln(rt))\\
  (rt)^{\rho(1-\mu_R)}e^{\rho\mu_I(\varphi+\psi)} e^{2\pi\rho\mu_I} \left[\cos(\rho(1-\mu_R)(\varphi+\psi)-\rho\mu_I\ln(rt))\cos(\rho(1-\mu_R)2\pi)\right. -\\ \left.\sin(\rho(1-\mu_R)(\varphi+\psi)-\rho\mu_I\ln(rt))\sin(\rho(1-\mu_R)2\pi)\right]=\\
  f_2^{Re}(r,\varphi,t,\psi) e^{2\pi\rho\mu_I}\cos(2\pi\rho(1-\mu_R)) -\\
  (rt)^{\rho(1-\mu_R)}e^{\rho\mu_I(\varphi+\psi)} e^{2\pi\rho\mu_I} \sin(\rho(1-\mu_R)(\varphi+\psi)-\rho\mu_I\ln(rt))\sin(2\pi\rho(1-\mu_R)).
\end{multline*}
From this it is clear that  $f_2^{Re}(r,\varphi+2\pi,t,\psi)=f_2^{Re}(r,\varphi,t,\psi)$, if $\mu_I=0$ and $\rho(1-\mu_R)=m$, where $m=0,\pm1,\pm2,\pm3,\dots$ (integer numbers). In the case, if $\mu_I\neq0$, then for any $\rho>1/2$ and any  $\mu_R$, we obtain  $f_2^{Re}(r,\varphi+2\pi,t,\psi)\neq f_2^{Re}(r,\varphi,t,\psi)$. If $\mu_I=0$, but $\rho(1-\mu_R)\neq m$, we also obtain that $f_2^{Re}(r,\varphi+2\pi,t,\psi)\neq f_2^{Re}(r,\varphi,t,\psi)$.

Similarly, for the function $f_2^{Im}(r,\varphi,t,\psi)$, we have
\begin{multline*}
  f_2^{Im}(r,\varphi+2\pi,t,\psi)=f_2^{Im}(r,\varphi,t,\psi)e^{2\pi\rho\mu_I}\cos(2\pi\rho(1-\mu_R))+ \\
  (rt)^{\rho(1-\mu_R)}e^{\rho\mu_I(\varphi+\psi)} e^{2\pi\rho\mu_I} \cos(\rho(1-\mu_R)(\varphi+\psi)-\rho\mu_I\ln(rt))\sin(2\pi\rho(1-\mu_R)).
\end{multline*}
From this one can see that  $f_2^{Im}(r,\varphi+2\pi,t,\psi)=f_2^{Im}(r,\varphi,t,\psi)$, if $\mu_I=0$ and $\rho(1-\mu_R)=m$, where $m=0,\pm1,\pm2,\pm3,\dots$ (integer numbers). In the case, if $\mu_I\neq0$, then for any $\rho>1/2$ and any $\mu_R$ we obtain  $f_2^{Im}(r,\varphi+2\pi,t,\psi)\neq f_2^{Im}(r,\varphi,t,\psi)$. If $\mu_I=0$, but $\rho(1-\mu_R)\neq m$, and also obtain that  $f_2^{Im}(r,\varphi+2\pi,t,\psi)\neq f_2^{Im}(r,\varphi,t,\psi)$.

Combining these facts, we obtain that for the function $f_2(\zeta,z)$
\begin{equation}\label{statement_f2}
  \zeta=0\ \mbox{is}
  \left\{\begin{array}{ll}
      \mbox{the regular point}, & \left\{\begin{array}{l}\mbox{at}\ \mu_I=0\ \mbox{and}\ \rho(1-\mu_R)=m_1,
                                                \\\mbox{where}\ m_1=0,1,2,3,\dots,
                                                \end{array}\right.\\
      \mbox{the pole of order}\ m_2, &\left\{\begin{array}{l} \mbox{at}\ \mu_I=0\ \mbox{and}\ \rho(1-\mu_R)=-m_2,\\
                                                \mbox{where}\ m_2=1,2,3,\dots,
                                                \end{array}\right.\\
      \mbox{a branch point}, &\left\{
      \begin{array}{l}
            \mbox{If}\ \mu_I\neq0\ \mbox{and any}\ \mu_R\ \mbox{and}\ \rho>1/2,\\
            \mbox{If}\ \mu_I=0, \mbox{then for any}\ \rho>1/2\ \mbox{and}\\
                        \mu_R\ \mbox{such that}\ \rho(1-\mu_R)\neq m.\end{array}\right.
  \end{array} \right.
\end{equation}
Here, as above, $m=0,\pm1,\pm2,\pm3,\dots$.

We consider now the function $f_3(\zeta)$. Apparently this function has one singular point $\zeta=1$ which is the pole of the first order. Using now this result as well as the statement (\ref{statement_f1}) and (\ref{statement_f2}) for the function (\ref{eq:Phi_product}) we obtain  the statement of the theorem.
\end{proof}

As one can see from the theorem which has just been proved, at values of parameters $\rho=1, \mu_I=0$ and $\mu_R=0,\pm1,\pm2,\pm3,\dots$ the point $\zeta=1$ is the pole of the first order and the point $\zeta=0$ is either a regular point or a pole. As a result, at such values of parameters the integral in the representaion (\ref{eq:MLF_int}) can be calculated analytically. We formulate the obtained result in the form
\begin{corollary}\label{coroll:MLF_case_rho=1}
For values of parameters $\rho=1$, $\delta_{1\rho}=\delta_{2\rho}=\pi$, any complex $z$, which is such that $\pi/2<\arg z<3\pi/2$ and for integer real values of the parameter $\mu=n, n=0,\pm1,\pm2,\pm3,\dots$ the Mittag-Leffler function has the form:
\begin{enumerate}
\item If $n\leqslant1$ (i. e. $n=1,0,-1,-2,-3,\dots$), then
\begin{equation}\label{eq:MLF_mu<=1}
  E_{1,n}(z)=e^z z^{1-n},
\end{equation}

\item If $n\geqslant2$ (i. e. $n=2,3,4,\dots$), then
\begin{equation}\label{eq:MLF_mu>=2}
  E_{1,n}(z)=z^{1-n}\left(e^z-\sum_{k=0}^{n-2}\frac{z^k}{k!}\right),
\end{equation}
\end{enumerate}
\end{corollary}

\begin{proof}
According to the theorem~\ref{lemm:MLF_int} and using (\ref{eq:Phi_fun}) for the parameter value $\rho=1$ and representation (\ref{eq:MLF_int}) will take the form
\begin{equation}\label{eq:MLF_int_rho=1}
  E_{1,\mu}(z)=\int_{\gamma_\zeta}\Phi_{1,\mu}(\zeta,z)d\zeta.
\end{equation}
In this case, the conditions for values of parameters $\delta_{1\rho}$ and $\delta_{2\rho}$ take the form $\pi/2<\delta_{1\rho}\leqslant\pi$, $\pi/2<\delta_{2\rho}\leqslant\pi$. Since these parameters can take arbitrary values of these intervals, we choose $\delta_{1\rho}=\delta_{2\rho}=\pi$. This choice is determined by the fact that it is that under such values of parameters $\delta_{1\rho}$ and $\delta_{2\rho}$ the contour $\gamma_\zeta$ closes up and we can calculate the integral in (\ref{eq:MLF_int_rho=1}). As a result, the condition (\ref{eq:z_cond_lemm}) takes the form $\pi/2<\arg z<3\pi/2$, and the contour of integration  (\ref{eq:loop_gammaZeta}) is written in the form
\begin{equation}\label{eq:loop_gammaZeta_rho=1}
  \gamma_\zeta=\left\{\begin{array}{ll}
                       S_1=&\{\zeta: \arg\zeta=-2\pi,\quad |\zeta|\geqslant 1+\epsilon\},\\
                       C_\epsilon=&\{\zeta: -2\pi\leqslant\arg\zeta\leqslant0,\quad |\zeta|=1+\epsilon\},\\
                       S_2=&\{\zeta: \arg\zeta=0,\quad|\zeta|\geqslant1+\epsilon\}.
                     \end{array}\right.
\end{equation}

We represent a complex parameter $\mu$ in the form $\mu=\mu_R+i\mu_I$ and we use the theorem~\ref{lemm:MLF_SingPoints}. According to this theorem, the function $\Phi_{1,\mu}(\zeta,z)$ at values $\mu_I=0$  and $\mu_R=1-m_1$, where $m_1=0,1,2,3,\dots$ has one singular point $\zeta=1$, which is a pole of the first order. The point $\zeta=0$, in this case, is the regular point. In case, if $\mu_I=0$ and $\mu_R=1+m_2$, where $m_2=1,2,3,\dots$  the function $\Phi_{1,\mu}(\zeta,z)$  has two singular points: the point $\zeta=1$ is a pole of the first order and the point $\zeta=0$ is a pole of order $m_2$. As one can see, in both cases the point $\zeta=0$ is not a branch point. As a result, in these two cases, the function $\Phi_{1,\mu}(\zeta,z)$ is the entire function of a complex variable $\zeta$. From here it follows that when $\mu_I=0$, and $\mu_R=0,\pm1,\pm2,\pm3,\dots$ ($\mu_R$ takes  integer numerical values an arc of the circle $C_\epsilon$, entering the contour (\ref{eq:loop_gammaZeta_rho=1}) is nothing more than a closed circle of radius $1+\epsilon$, and the half-lines $S_1$ and $S_2$ go along the positive part of a real axis in mutually opposite directions. With all other values of the parameter $\mu$ (when $\mu_I\neq0$ or $\mu_R$ is not integer according to the theorem~\ref{lemm:MLF_SingPoints}, the point $\zeta=0$ is a branch point of the function $\Phi_{1,\mu}(\zeta,z)$. As a result, the circle $C_\epsilon$ of the contour (\ref{eq:loop_gammaZeta_rho=1}) will not close up and the half-lines $S_1$ and $S_2$ will go along the low and upper banks of the cut of  the complex plane  which passes along the positive part of a real axis.

It follows from the foregoing that in case when $\mu_I=0$ and $\mu_R=0,\pm1,\pm2,\pm3,\dots$ the integral in (\ref{eq:MLF_int_rho=1}), can be calculated with the use of the theory of residues. In fact, in this case  (\ref{eq:MLF_int_rho=1}) takes the form
\begin{equation}\label{eq:MLF_int_rho=1_mu=n}
  E_{1,\mu}(z)=\int_{S_1}\Phi_{1,\mu}(\zeta,z) d\zeta+\int_{C_\epsilon}\Phi_{1,\mu}(\zeta,z) d\zeta +\int_{S_2}\Phi_{1,\mu}(\zeta,z) d\zeta= \int_{C_\epsilon}\Phi_{1,\mu}(\zeta,z) d\zeta,
\end{equation}
since $\int_{S_1}\Phi_{1,\mu}(\zeta,z) d\zeta+\int_{S_2}\Phi_{1,\mu}(\zeta,z) d\zeta=0$. We remind, here $C_\epsilon$ is the circle of radius $1+\epsilon, \epsilon>0$, with the center in the point $\zeta=0$.

We consider the case $\mu_I=0, \mu_R=1-m_1, m_1=0,1,2,3,\dots$. In this case the function $\Phi_{1,\mu}(\zeta,z)$ has one singular point $\zeta=1$ which is the pole of the first order. According to the principal residue theorem we have
\begin{equation*}
  \int_{C_\epsilon}\Phi_{1,\mu}(\zeta,z) d\zeta=
  2\pi i\,\mbox{Res}\left[\frac{1}{2\pi i}\frac{e^{\zeta z}(\zeta z)^{m_1}}{\zeta-1}, \zeta=1\right]=e^zz^{m_1},\quad m_1=0,1,2,3,\dots
\end{equation*}
Using this result in (\ref{eq:MLF_int_rho=1_mu=n}), we obtain
\begin{equation*}
  E_{1,1-m_1}(z)=e^zz^{m_1},\quad m_1=0,1,2,3,\dots.
\end{equation*}
Taking into consideration that $\mu=1-m_2$, we obtain that the parameter $\mu$ can take only real  integer number values. Designating $\mu\equiv n$, we definitely obtain
\begin{equation*}
  E_{1,n}(z)=e^zz^{1-n},\quad n=1,0,-1,-2,-3,\dots.
\end{equation*}

Now we consider the case $\mu_I=0, \mu_R=1+m_2, m_2=1,2,3,\dots$. In this case the function $\Phi_{1,\mu}(\zeta,z)$ has two singular points: the point $\zeta=1$ is a pole of the first order and the point $\zeta=0$ is a pole of order $m_2$. According to the principal residue theorem for (\ref{eq:MLF_int_rho=1_mu=n}) we obtain
\begin{equation}\label{eq:MLF_case_rho=1_mu>1}
  E_{1,1+m_2}(z)=\int_{C_\epsilon}\Phi_{1,1+m_2}(\zeta,z)d\zeta\\
  =2\pi i\, \left(\mbox{Res}\left[\Phi_{1,1+m_2}(\zeta,z),\zeta=0\right] + \mbox{Res}\left[\Phi_{1,1+m_2}(\zeta,z),\zeta=1\right]\right).
\end{equation}

The residue in the point $\zeta=1$ has the form
\begin{equation}\label{eq:Phi_res_zeta=1}
  \mbox{Res}\left[\Phi_{1,1+m_2}(\zeta,z),\zeta=1\right]=
  \frac{1}{2\pi i}\lim_{\zeta\to1}(\zeta-1)\frac{e^{\zeta z} (\zeta z)^{-m_2}}{\zeta-1}=\frac{1}{2\pi i} e^z z^{-m_2}.
\end{equation}

For the residue in the point $\zeta=0$ we have
\begin{equation}\label{eq:Phi_res_zeta=0}
  \mbox{Res}\left[\Phi_{1,1+m_2}(\zeta,z),\zeta=0\right]
  = \frac{z^{-m_2}}{2\pi i}\frac{1}{(m_2-1)!}\lim_{\zeta\to0}\frac{d^{m_2-1}}{d\zeta^{m_2-1}}\frac{e^{\zeta z}}{\zeta-1},\quad m_2=1,2,3,\dots.
\end{equation}
To calculate the derivative of order  $m_2-1$ for an arbitrary value $m_2$, we use the method of mathematical induction. For $m_2=1$, we obtain
\begin{equation*}
  \frac{d^0}{d\zeta^0}\left[ \frac{e^{\zeta z}}{\zeta-1}\right]= \frac{e^{\zeta z}}{\zeta-1}.
\end{equation*}
For $m_2=2$
\begin{equation*}
  \frac{d}{d\zeta} \left[\frac{e^{\zeta z}}{\zeta-1}\right]= e^{\zeta z}\left(\frac{z}{\zeta-1}-\frac{1}{(\zeta-1)^2}\right).
\end{equation*}
For $m_2=3$
\begin{equation*}
  \frac{d^2}{d\zeta^2} \left[\frac{e^{\zeta z}}{\zeta-1}\right]= e^{\zeta z}\left(\frac{z^2}{\zeta-1}-\frac{2z}{(\zeta-1)^2}+ \frac{2}{(\zeta-1)^3}\right).
\end{equation*}
Continuing in a similar way, we can understand that for an arbitrary $m_2$ the derivative has the form
\begin{equation}\label{eq:derivative_Phi_n}
  \frac{d^{m_2-1}}{d\zeta^{m_2-1}} \left[\frac{e^{\zeta z}}{\zeta-1}\right]=
  e^{\zeta z}\sum_{k=0}^{m_2-1}(-1)^{m_2-1-k}\frac{(m_2-1)!}{k!}\frac{z^k}{(\zeta-1)^{m_2-k}}.
\end{equation}

Let us assume that this formula is true for some $m_2=n$. According to the method of mathematical induction, if this formula turns out to be true for $m_2=n+1$ as well, then it will be true for an arbitrary value $m_2$. Let $m_2=n+1$, then
\begin{multline}\label{eq:derivative_Phi}
  \frac{d^n}{d\zeta^n} \left[\frac{e^{\zeta z}}{\zeta-1}\right]=\frac{d}{d\zeta}\left(\frac{d^{n-1}}{d\zeta^{n-1}} \left[\frac{e^{\zeta z}}{\zeta-1}\right]\right)
  =\frac{d}{d\zeta} \left(e^{\zeta z}\sum_{k=0}^{n-1}(-1)^{n-1-k}\frac{(n-1)!}{k!}\frac{z^k}{(\zeta-1)^{n-k}}\right)\\
  =e^{\zeta z}\left(\sum_{k=0}^{n-1}(-1)^{n-1-k}\frac{(n-1)!}{k!}\frac{z^{k+1}}{(\zeta-1)^{n-k}}
  + \sum_{k=0}^{n-1}(-1)^{n-1-k}\frac{(n-1)!}{k!}\frac{(-1)(n-k)}{(\zeta-1)^{n-k+1}}z^k\right).
\end{multline}
Next, it is necessary to expand the resulting two sums. For this, we write the summands with $k=0,1,\dots,k-1,\dots,n-2,n-1$ from the first sum, and from the second sum the summands with $k=0,1,2,\dots,k,\dots,n-1$. Such a selection of values $k$ in each of the sums is determined by the fact that we need to have in the resulting sum only  the summands with the equal power $z$ and $(\zeta-1)$. As a result, we obtain
\begin{multline}\label{eq:derivative_Phi_tmp1}
  \frac{d^n}{d\zeta^n} \left[\frac{e^{\zeta z}}{\zeta-1}\right]=
  e^{\zeta z}\left((-1)^{n-1}\frac{(n-1)!}{0!}\frac{z}{(\zeta-1)^n}+  (-1)^{n-2}\frac{(n-1)!}{1!}\frac{z^2}{(\zeta-1)^{n-1}}\right.\\
  +\dots+(-1)^{n-k}\frac{(n-1)!}{(k-1)!}\frac{z^k}{(\zeta-1)^{n-k+1}}+\dots+
  (-1)\frac{(n-1)!}{(n-2)!}\frac{z^{n-1}}{(\zeta-1)^{2}}\\
  + (-1)^{0}\frac{(n-1)!}{(n-1)!}\frac{z^n}{(\zeta-1)^1}+
  (-1)^{n-1}\frac{(n-1)!}{0!}\frac{(-n)}{(\zeta-1)^{n+1}}\\
  +(-1)^{n-2}\frac{(n-1)!}{1!}\frac{(-1)(n-1)}{(\zeta-1)^{n}}z+
  (-1)^{n-3}\frac{(n-1)!}{2!}\frac{(-1)(n-2)}{(\zeta-1)^{n-1}}z^2+\dots+\\
  \left.(-1)^{n-k-1}\frac{(n-1)!}{k!}\frac{(-1)(n-k)}{(\zeta-1)^{n-k+1}}z^k+\dots+
  (-1)^{0}\frac{(n-1)!}{(n-1)!}\frac{(-1)}{(\zeta-1)^{2}}z^{n-1}\right).
\end{multline}

Now combining the summands with equal powers $z$ in this sum, it is clear that the summand with index $k$ of the first sum in (\ref{eq:derivative_Phi}) is combined with the summand with index $k+1$ of the second sum.  At the same time, the summands with $z^0$ and $z^n$ there is no pairwise summand. As a result, for the summand with the multiplier $z^k$ we have
\begin{multline*}
  (-1)^{n-k}\frac{(n-1)!}{(k-1)!}\frac{z^k}{(\zeta-1)^{n-k+1}}+
  (-1)^{n-k-1}\frac{(n-1)!}{k!}\frac{(-1)(n-k)}{(\zeta-1)^{n-k+1}}z^k\\
  =(-1)^{n-k}(n-1)!\frac{z^k}{(\zeta-1)^{n-k+1}}\left(\frac{1}{(k-1)!}+\frac{n-k}{k!}\right)\\
  =(-1)^{n-k}(n-1)!\frac{z^k}{(\zeta-1)^{n-k+1}}\left(\frac{k+n-k}{k!}\right)=
  (-1)^{n-k}\frac{n!}{k!}\frac{z^k}{(\zeta-1)^{n-k+1}}.
\end{multline*}
As a result, the expression (\ref{eq:derivative_Phi_tmp1}) takes the form
\begin{multline*}
  \frac{d^n}{d\zeta^n} \left[\frac{e^{\zeta z}}{\zeta-1}\right]=
  e^{\zeta z}\left((-1)^n\frac{n!}{0!}\frac{1}{(\zeta-1)^{n+1}}+
  (-1)^{n-1}\frac{n!}{1!}\frac{z}{(\zeta-1)^n}
  +  (-1)^{n-2}\frac{n!}{2!}\frac{z^2}{(\zeta-1)^{n-1}}+\dots\right.\\
  +(-1)^{n-k}\frac{n!}{k!}\frac{z^k}{(\zeta-1)^{n-k+1}}+\dots
  +(-1)\frac{n!}{(n-1)!}\frac{z^{n-1}}{(\zeta-1)^{2}}+
  \left.(-1)^{0}\frac{n!}{n!}\frac{z^n}{(\zeta-1)}\right)=\\
  e^{\zeta z}\sum_{k=0}^{n}(-1)^{n-k}\frac{n!}{k!}\frac{z^k}{(\zeta-1)^{n-k+1}}.
\end{multline*}
From this it is clear that the expression obtained is nothing more than the expression (\ref{eq:derivative_Phi_n}) written for $m_2=n+1$. Consequently, the validity of the formula(\ref{eq:derivative_Phi_n}) is proved.

Now using (\ref{eq:derivative_Phi_n}) in (\ref{eq:Phi_res_zeta=0}) we obtain
\begin{equation*}
  \mbox{Res}\left[\Phi_{1,1+m_2}(\zeta,z),\zeta=0\right]=
  -\frac{z^{-m_2}}{2\pi i}\sum_{k=0}^{m_2-1}\frac{z^k}{k!}.
\end{equation*}
Now putting this expression and (\ref{eq:Phi_res_zeta=1}) in (\ref{eq:MLF_case_rho=1_mu>1}) we get
\begin{equation*}
  E_{1,1+m_2}(z)=2\pi i\left(\frac{e^z z^{-m_2}}{2\pi i}-\frac{z^{-m_2}}{2\pi i}\sum_{k=0}^{m_2-1}\frac{z^k}{k!}\right)=
  \frac{1}{z^{m_2}}\left(e^z-\sum_{k=0}^{m_2-1}\frac{z^k}{k!}\right),
\end{equation*}
where $m_2=1,2,3,\dots$.  Since in the considered case $\mu=1+m_2$, then this means that $\mu$ is real and takes only  integer number values.  Designating $\mu\equiv n$, we definitely obtain
\begin{equation*}
  E_{1,n}(z)=z^{1-n}\left(e^z-\sum_{k=0}^{n-2}\frac{z^k}{k!}\right),\quad n=2,3,4,\dots.
\end{equation*}
\end{proof}

The proved corollary gives the expression for the Mittag-Leffler function in the case $\rho=1$ and  integer number values of the parameter $\mu$. It should be noted that the formulas (\ref{eq:MLF_mu<=1}) and (\ref{eq:MLF_mu>=2}) can be obtained directly from the formula (\ref{eq:MLF_gen}). We will formulate this result as a remark.

\begin{remark}\label{rem:MLF_case_rho=1}
  Let $\rho=1$, and $\mu=n$, where $n$ - integer numbers ($n=0,\pm1,\pm2,\pm3,\dots$). Using these values in  (\ref{eq:MLF_gen}) and substituting the summation index $k+n\to m+1$, we obtain  \begin{equation}\label{eq:MLF_sum_any_n}
    E_{1,n}(z)=\sum_{k=0}^{\infty}\frac{z^k}{\Gamma(k+n)}=z^{1-n}\sum_{m=n-1}^{\infty}\frac{z^m}{\Gamma(m+1)}.
  \end{equation}
  From this expression one can see that it is necessary to consider three cases $n<1,n=1,n>1$.

In fact, in the case $n<1$ the lower limit of the sum is negative. This leads to the appearance of summands with a negative argument in the sum in the function $\Gamma(m+1)$.  We represent  $n=\sgn(n)|n|$, where
  \begin{equation*}
    \sgn(x)=\left\{\begin{array}{cc}
                      1, & x\geqslant0, \\
                      -1, & x<0.
                    \end{array}\right.
  \end{equation*}
  Dividing now in (\ref{eq:MLF_sum_any_n}) the sum by two sums one of which with the summation index $m$ that takes only negative values but the second sum has non-negative values, we get
  \begin{equation*}
    E_{1,n}(z)= z^{1-\sgn(n)|n|}\hspace{-5mm}\sum_{m=\sgn(n)|n|-1}^{\infty}\frac{z^m}{\Gamma(m+1)} = z^{1-\sgn(n)|n|} \left(\sum_{m=\sgn(n)|n|-1}^{-1}\frac{z^m}{\Gamma(m+1)} + \sum_{m=0}^{\infty}\frac{z^m}{\Gamma(m+1)}\right),
  \end{equation*}
  where $n<1$. Thus, the first sum in brackets contains the summands that have $m+1=0,-1,-2,-3,\dots$ at $n=0,-1,-2,-3$, but the second sum contains the summands only with positive values  $m+1$. Taking into consideration that  $1/\Gamma(x)=0$, if  $x=0,-1,-2,-3,\dots$ and $\Gamma(m+1)=m!$, we obtain
  \begin{equation}\label{eq:MLF_sum_n<1}
    E_{1,n}(z)=z^{1-n}e^z, \quad n<1.
  \end{equation}

  In case $n=1$ from (\ref{eq:MLF_sum_any_n}) it directly follows
  \begin{equation}\label{eq:MLF_sum_n=1}
    E_{1,1}(z)=e^z.
  \end{equation}
  Now combining the formulas (\ref{eq:MLF_sum_n<1}) and (\ref{eq:MLF_sum_n=1}) we obtain the formula (\ref{eq:MLF_mu<=1}).

  We consider the case  $n>1$. Adding and subtracting in (\ref{eq:MLF_sum_any_n}) the summand $z^{1-n}\sum_{m=0}^{n-2}z^m/\Gamma(m+1)$. we get
  \begin{equation*}
    E_{1,n}(z)=z^{1-n}\left(\sum_{m=n-1}^{\infty}\frac{z^m}{\Gamma(m+1)}+\sum_{m=0}^{n-2}\frac{z^m}{\Gamma(m+1)}- \sum_{m=0}^{n-2}\frac{z^m}{\Gamma(m+1)}\right)=z^{1-n}\left(e^z-\sum_{m=0}^{n-2}\frac{z^m}{m!}\right).
  \end{equation*}
  Thus, the obtained result coincides completely with (\ref{eq:MLF_mu>=2}).
\end{remark}

The coincidence of the results of the remark that has just been given and results of the corollary~\ref{coroll:MLF_case_rho=1} confirms the validity of the lemma~\ref{lemm:MLF_int}.

It should be pointed out that the corollary~\ref{coroll:MLF_case_rho=1} exhausts all possible situations when one can get an explicit form for the Mittag-Leffler function at $\rho>1/2$. In fact, as it follows from the theorem~\ref{lemm:MLF_SingPoints} If $\rho$ is not  integer or $\mu_I\neq0$ or $\mu_R$ is not an integer number, then the point $\zeta=0$ will be a branch point of the integrand (\ref{eq:MLF_int}). If the conditions of the theorem~\ref{lemm:MLF_SingPoints} are met (i. e. $\rho$ and $\mu_R$ are integer numbers and $\mu_I=0$), but $\rho>1$, then in this case the contour $\gamma_\zeta$ defined by (\ref{eq:loop_gammaZeta}) will not close up. Therefore, in all remaining cases of parameter values $\rho, \mu, \delta_{1\rho}, \delta_{2\rho}$ that do not satisfy the condition of the corollary~\ref{coroll:MLF_case_rho=1} it is necessary to use the representation (\ref{eq:MLF_int}).

\section{Conclusion}

At the moment, the most used form of the integral representation of the function $E_{\rho,\mu}(z)$ is the representation (\ref{eq:MLF_int_Dzhr}). This representation was obtained for the first time in the work \cite{Dzhrbashyan1954_eng} and later was introduced in the book \cite{Dzhrbashyan1966_eng}. Using this integral representation in the papers \cite{Dzhrbashyan1954_eng,Dzhrbashyan1966_eng} asymptotic properties were investigated and a number of asymptotic formulas were obtained. Later, this integral representation formed the basis of a number of works  \cite{Gorenflo1991, Gorenflo2002b,Seybold2005,Hilfer2006},  and used in the books \cite{Dzhrbashian1993,Gorenflo2014}. However, a detailed study of the proof presented in the work  \cite{Dzhrbashyan1954_eng} and in the book \cite{Dzhrbashyan1966_eng} (see chapter 3, \S 2, lemma 3.2.1 in \cite{Dzhrbashyan1966_eng}), showed that  the integral representation (\ref{eq:MLF_int_Dzhr}) is not true. \emph{In this integral representation, the case $z\in G^{(+)}(\varepsilon,\delta)$ is erroneous.  The results related to the case $z\in G^{(-)}(\varepsilon,\delta)$ will be true.} A more detailed description of the essence of the problem and a description of the error can be found in the work \cite{Saenko2020}. In this regard, the results of the abovementioned works should be used with caution, as well as the results of other works in which the representation (\ref{eq:MLF_int_Dzhr}) is used.
To eliminate the mistake made in the works \cite{Dzhrbashyan1954_eng,Dzhrbashyan1966_eng} in the works by \cite{Saenko2019, Saenko2020} a theorem was formulated and proved that gives a correct integral representation for the Mittag-Leffler function. In this paper, this result is presented in the lemma~\ref{lemm:MLF_int}. The proof of this lemma can be found in the work \cite{Saenko2020}. This paper presents the results of a study of some properties of the integral representation  (\ref{eq:MLF_int}). In particular, the singular points of this integral representation have been defined. It has been obtained that the point $\zeta=1$ is a pole of the first order, and the point $\zeta=0$, depending on parameter values $\rho$ and $\mu$, can be either a pole, or a branch point, or a regular point. Further studies have shown that at value $\rho=1$ and integer $\mu$   the integral in (\ref{eq:MLF_int}) can be calculated  using the theory of residues. The result obtained is formulated in the corollary~\ref{coroll:MLF_case_rho=1}. Direct calculation of the sum in (\ref{eq:MLF_gen}) at $\rho=1$ and integer real $\mu$ shows that the result obtained coincides completely with the result of the corollary~\ref{coroll:MLF_case_rho=1}. This fact indicates the correctness of the integral representation of the function $E_{\rho,\mu}(z)$ formulated in the lemma~\ref{lemm:MLF_int}.

As further plans for the study of the integral representation (\ref{eq:MLF_int}), it is planned to complete the transition from integration over a complex variable to integration over real variables. This will allow us to represent the contour integral as the sum of definite and improper integrals and, using numerical integration methods, calculate these integrals. However, all this requires additional research beyond the scope of this work.

\AcknowledgementSection
This work was supported by the Russian Foundation for Basic Research (projects No 19-44-730005, 18-51-53018 and 20-07-00655).

The author thanks to M.~Yu.~Dudikov for translation the article into English.

\bibliographystyle{unsrt}
\bibliography{d:/bibliography/library}

\end{document}